\documentclass{svproc}
\usepackage{url}
\usepackage{amsmath}
\usepackage{amssymb}
\usepackage{enumitem}
\usepackage{hyperref}[hidelinks]
\newtheorem{observation}[theorem]{Observation}

\def\cont{^\frown}
\def\L{\mathrm L}
\def\R{\mathrm R}
\def\X{\mathrm X}
\def\str#1{\mathbf {#1}}
\def\lexleq{{\leq_{\mathrm{lex}}}}
\def\lexlt{{<_{\mathrm{lex}}}}
\def\Alphabet{\Sigma}

\def\AmbStr#1{(#1,\allowbreak \lexleq\nobreak,\allowbreak\preceq\nobreak,\allowbreak\perp)}

\begin{document}
\mainmatter              
\title{Big Ramsey degrees of the generic partial order}
\titlerunning{Big Ramsey degrees of partial orders}  
%
\author{Martin Balko\inst{1} \and David Chodounsk\'y\inst{1} \and Natasha Dobrinen\inst{2} \and Jan Hubi\v cka\inst{1} \and\\
	Mat\v ej Kone\v cn\'y\inst{1} \and Lluis Vena\inst{3} \and Andy Zucker\inst{4}}
\authorrunning{Martin Balko et al.} 
%
%
\institute{Department of Applied Mathematics (KAM), Charles University, Ma\-lo\-stranské~nám\v estí 25, Praha 1, Czech Republic,\\
	\email{\{balko,chodounsky,hubicka,matej\}@kam.mff.cuni.cz},
	\and
	Department of Mathematics, University of Denver, C.M. Knudson Hall 302, Denver, USA,\\
	\email{natasha.dobrinen@du.edu}
	\and
	Universitat Polit\`ecnica de Catalunya, Barcelona, Spain,\\
	\email{lluis.vena@gmail.com}
	\and
	University of California San Diego, USA,\\
	\email{azucker@ucsd.edu}
}
\maketitle              

\begin{abstract}
	As a result of 33 intercontinental Zoom calls,
	we characterise big Ramsey degrees of the generic partial order in a
	similar way as Devlin characterised big Ramsey degrees of
	the generic linear order (the order of rationals).
\end{abstract}
\section{Introduction}
Given partial orders $\str{A}$ and $\str{B}$, we denote by $\binom{\str{B}}{\str{A}}$ the set
of all embeddings from $\str{A}$ to $\str{B}$. We write $\str{C}\longrightarrow (\str{B})^\str{A}_{k,l}$ to denote the following statement:
for every colouring $\chi$ of $\binom{\str{C}}{\str{A}}$ with $k$ colours, there exists
an embedding $f\colon\str{B}\to\str{C}$ such that $\chi$ does not take more than $l$ values on $\binom{f(\str{B})}{\str{A}}$.
For a countably infinite partial order $\str{B}$ and its finite suborder $\str{A}$, the \emph{big Ramsey degree} of $\str{A}$ in $\str{B}$ is
the least number $l\in \mathbb N\cup \{\infty\}$ such that $\str{B}\longrightarrow (\str{B})^\str{A}_{k,l}$ for every $k\in \mathbb N$.

A partial order is \emph{homogeneous} if every isomorphism between
two of its finite
suborders extends to an automorphism.  It is well known that up to
isomorphism there is a unique homogeneous partial order $\str{P}=(P,\leq_\str{P})$ such that
every countable partial order has an embedding to $\str{P}$. The order $\str{P}$ is called
the \emph{generic partial order}. We refine the following recent result.

\begin{theorem}[Hubi\v cka~\cite{Hubicka2020CS}]
	The big Ramsey degree of every finite partial order
	in the
	generic order $\str{P}$ is finite.
\end{theorem}

We characterise the big Ramsey degrees of $\str{P}$ using special sets
of finite words. Our characterisation also leads to a big Ramsey structure with applications in topological dynamics~\cite{zucker2017}.

Currently, there are only relatively few classes of structures where big Ramsey degrees are fully understood.
The Ramsey theorem implies that the big Ramsey degree of every finite linear order in the order of $\omega$ is 1.
In 1979, Devlin refined upper bounds by Laver and characterised big Ramsey degrees of the order of rationals~\cite{devlin1979,todorcevic2010introduction}.  Laflamme, Sauer and Vuksanovi\'c characterised big Ramsey degrees of the Rado (or random) graph and related random structures in binary languages~\cite{Laflamme2006}.
Recently, a characterisation of big Ramsey degrees of the triangle-free Henson graph was obtained by Dobrinen~\cite{dobrinen2020ramsey} and independently by the remaining authors of this abstract. See also~\cite{coulson2020substructure}.
Here, we characterize the structures that give the big Ramsey degrees for the generic partial order.

Our construction makes use of the following partial order introduced in~\cite{Hubicka2020CS} (which is closely tied to
the
order of 1-types within a fixed enumeration of $\str{P}$).
Let $\Sigma=\{\L,\X,\R\}$ be an \emph{alphabet} ordered by $\lexlt$ as
$\L\lexlt \X\lexlt \R$.  We denote by
$\Sigma^*$ the set of all finite words in the alphabet $\Sigma$, by
$\lexleq$ their lexicographic order, and by $|w|$ the length of the word $w$ (whose characters are indexed by natural numbers starting at $0$).
For $w,w'\in \Sigma^*$, we set $w\prec w'$ if and only if there exists $i$ such that $0\leq i<\min(|w|,|w'|)$, $(w_i,w'_i)=(\L,\R)$, and
$w_j\leq_{\mathrm{lex}}w'_j$ for every $0\leq j< i$.
It is not difficult to check
that $(\Sigma^*,\preceq)$ is a partial order and that $(\Sigma^*,\lexleq)$ is one of its linear extensions~\cite{Hubicka2020CS}.

We write $w\perp w'$ if $w_i\lexlt w'_i$ and $w'_{j}\lexlt w_{j}$ for some $0\leq i,j< \min{(|w|,|w'|)}$.
Note that it is not necessarily true that $(w\not\preceq w'\land w'\not\preceq w)\iff w\perp w'$, as, for example, $\L\R\preceq \R\L$, $\L\R\perp \R\L$, $\X\not\preceq \R$, $\X\not\perp \R$.
However, we will construct subsets of $\Sigma^*$ where this is satisfied.
We call words $w$ and $w'$ \emph{related} if one of expressions $w\preceq w'$, $w'\preceq w$ or $w\perp w'$ holds, otherwise they are \emph{unrelated}.

Given a word $w$ and an integer $i \geq 0$, we denote by $w|_i$ the \emph{initial segment} of $w$ of length $i$.
For $S\subseteq \Sigma^*$, we let $\overline{S}$ be the set $\{w|_i\colon w\in S, 0\leq i\leq |w|\}$.
The set $\Sigma^*$ can be seen as a rooted ternary tree and sets $S=\overline{S}\subseteq\Sigma^*$ as its subtrees.
Given $i\geq 0$, we let $S_i=\{w\in S:|w|=i\}$ and call it the \emph{level $i$ of $S$}.
A word $w\in S$ is called a \emph{leaf} of $S$ if there is no word $w'\in S$ extending $w$.
Let $L(S)$ be the set of all leafs of $S$. Given a word $w$ and a character $c\in \Sigma$, we denote by $w\cont c$
the word created from $w$ by adding $c$ to the end of $w$. We also set $S\cont c=\{w\cont c:w\in S\}$.

To characterise big Ramsey degrees of $\str{P}$, we introduce the following  technical definition, whose intuitive meaning is explained at the end of this section.

\begin{definition}
	\label{def:balko}
	A set $S\subseteq \Sigma^*$ is called a \emph{poset-type} if $S=\overline{S}$ and precisely one of the following four conditions is satisfied for every $i$ with $0\leq i< \max_{w\in S}|w|$:
	\begin{enumerate}
		\item \textbf{Leaf:}  There is $w\in S_i$ related to every $u\in S_i\setminus\{w\}$ and $S_{i+1}=(S_i\setminus \{w\} )\cont \X.$
		\item \textbf{Branching:}  There is $w\in S_i$ such that $$S_{i+1}=\{z\in S_i:z\lexlt w\}\cont \X\cup\{w\cont \X,w\cont\R\}\cup \{z\in S_i:w\lexlt z\}\cont \R.$$
		\item \textbf{New $\perp$:} There are unrelated words $v\lexlt w\in S_i$ such that
		      \begin{eqnarray*}
			      S_{i+1}&=&\{z\in S_i:z\lexlt v\}\cont \X
			      \cup \{v\cont \R\}\\
			      &&\cup\{z\in S_i:v\lexlt z\lexlt w\hbox{ and }z\perp v\}\cont \X\\
			      &&\cup\{z\in S_i:v\lexlt z\lexlt w\hbox{ and } z\not\perp v\}\cont \R\\
			      &&\cup \{{w}\cont \X\}\cup \{z\in S_i:w\lexlt z\}\cont \R.
		      \end{eqnarray*}
		      Moreover, the following assumption is satisfied:
		      \begin{enumerate}[label=(\Alph*)]
			      \item\label{A2} For every $u\in S_i$, $v\lexlt u\lexlt w$ implies that at least one of $u\perp v$ or $u\perp w$ holds.
		      \end{enumerate}
		\item \textbf{New $\preceq$:}  There are unrelated words $v\lexlt w\in S_i$ such that
		      \begin{eqnarray*}
			      S_{i+1}&=&\{z\in S_i:z\lexlt v\hbox{ and }z\perp v\}\cont \X \cup \{z\in S_i:z\lexlt v\hbox{ and }z\not \perp v\}\cont \L\\
			      &&\cup \{v\cont \L\}\cup\{z\in S_i:v\lexlt z\lexlt w\}\cont \X\cup \{{w}\cont \R\}\\
			      &&\cup \{z\in S_i:w\lexlt z\hbox{ and }w\perp z\}\cont\X\\
			      &&\cup \{z\in S_i:w\lexlt z\hbox{ and }w\not\perp z\}\cont\R.
		      \end{eqnarray*}
		      Moreover, the following assumptions are satisfied:
		      \begin{enumerate}[label=(B\arabic*)]
			      \item\label{B1} For every $u\in S_i$ such that $u\lexlt v$, at least one of $u\preceq w$ or $u\perp v$ holds.
			      \item\label{B2} For every $u\in S_i$ such that $w\lexlt u$, at least one of $v\preceq u$ or $w\perp u$ holds.
		      \end{enumerate}
	\end{enumerate}
\end{definition}

Given a finite partial order $\str{Q}$, we let $T(\str{Q})$ be the set of all poset-types $S$ such that $(L(S),\preceq)$ is isomorphic to $\str{Q}$.
As our main result, we determine the big Ramsey degrees of $\str{P}$.

\begin{theorem}\label{thm:main}
	For every finite partial order $\str{Q}$, the big Ramsey degree of $\str{Q}$ in the generic partial order $\str{P}$ equals $|T(\str{Q})|\cdot |\mathrm{Aut}(\str{Q})|$.
\end{theorem}

Here, we only outline the main constructions giving the lower bound on big Ramsey degrees of $\str{P}$; see Section~\ref{sec:lowerBound}.
The upper bound follows from a refinement of~\cite{Hubicka2020CS} and will appear in the full version of this abstract.

Poset-types, which can be compared to \emph{Devlin types}~\cite{todorcevic2010introduction}, have a relatively intuitive meaning that we now outline.
Words $u\lexleq v$ are \emph{compatible} if there is no $i<\min(|u|,|v|)$ such that $(u_i,v_i)=(\R,\L)$, and if there exists $j<\min(|u|,|v|)$ such that $(u_j,v_j)=(\L,\R)$ then $u\prec v$
and $u\not\perp v$. Conditions~\ref{A2}, \ref{B1} and \ref{B2} from Definition~\ref{def:balko} originate from the following properties of compatible words.

\begin{proposition}
	\label{prop:perp}
	Let $u\lexleq v\lexleq w\lexleq z$ be mutually compatible words from $\Sigma^*_i$ for some $i\in \omega$.
	\begin{enumerate}
		\item If $v\prec w$ then the following two conditions are satisfied: (a) at least one of $u\prec w$ and $u\perp v$ holds;  (b) at least one of $v\prec z$ and $w\perp z$ holds.
		\item If $u\perp z$ and $u\lexlt v\lexlt z$ then at least one of $u\perp v$ and $v\perp z$ is satisfied.
	\end{enumerate}
\end{proposition}
\begin{proof}
	To prove~1(a), we choose $j$ such that $(v_j,w_j)=(\L,\R)$.  Now, if $u_j=L$, we have $u_j\prec w_j$
	by compatibility. If $u_j\in \{\X,\R\}$, we have $u\perp v$, since $u\lexlt v$. Part~1(b) follows from symmetry.

	For part~2, we choose $j$ such that $z_j\lexlt u_j$. If $u$ is unrelated to $v$ or if $u\preceq v$, we have $z_j\lexlt u_j\lexleq v_j$ (by compatibility) and, since $v\lexlt z$, it follows that $v\perp z$.
	The case when $v$ is unrelated to $z$ or $v\preceq z$ follows analogously.\qed
\end{proof}

One can view a poset-type $S$ as a binary branching tree and each level $S_i$ as a structure $\str{S}_i=\AmbStr{S_i}$.
It follows that all words in $S$ are mutually compatible and comparing with Proposition~\ref{prop:perp}
one can verify that if a level $S_i$ is a leaf level, then $\str{S}_{i+1}$ is isomorphic to $\str{S}_i$ with one vertex removed. If a level $S_i$ is a branching level, then $\str{S}_{i+1}$ is isomorphic to $\str{S}_i$ with one vertex $w\in S_i$ duplicated to $w\cont \X,w\cont \R\in S_{i+1}$. Observe also that $w\cont \X$ and $w\cont \R$ are unrelated. If a level $S_i$ has new $\preceq$ (or $\perp$), then $\str{S}_{i+1}$  is isomorphic to $\str{S}_i$ extended by one pair in the relation $\preceq$ (or $\perp$).

\section{The lower bound}
\label{sec:lowerBound}

Without loss of generality, we can assume that $P=\omega$
and thus we fix an (arbitrary) enumeration of $\str{P}$.
We define a function $\varphi\colon \omega\to \Alphabet^*$ by mapping $j\in P$ to a word $w$ of length $2j+2$ defined by
putting $(w_{2j},w_{2j+1})=(\L,\R)$ and, for every $i<j$, $(w_{2i},w_{2i+1})$ to $(\L,\L)$ if $j\leq_\str{P} i$, $(\R,\R)$ if $i\leq_\str{P} j$ and $(\X,\X)$ otherwise.
We set $T=\overline{\varphi[P]}$. The following result is easy to prove by induction.
\begin{proposition}[Proposition 4.7 of \cite{Hubicka2020CS}]
	The function $\varphi$ is an embedding $\str{P}\to (\Sigma^*,\preceq)$.
	More\-over, $\varphi(v)$ is a leaf of $T$ for every $v\in P$, all words in $T$ are mutually compatible, and if $v,w\in P$ are incomparable, we have $\varphi(v)\perp\varphi(w)$.
\end{proposition}

We will need the following refinement of this embedding.

\begin{theorem}
	\label{thm:posetemb}
	There exists an embedding $\psi\colon \str{P}\to (\Sigma^*,\preceq)$ such that
	$Q=\overline{\psi[\omega]}$ is a poset-type
	and $\psi(i)$ is a leaf of $Q$ for every $i\in P$.
\end{theorem}
\begin{proof}
	We proceed by induction on levels of $T$.
	For every $\ell$, we define an integer $N_\ell$ and a function $\psi_\ell\colon T_\ell\to \Sigma^*_{N_\ell}$.
	We will maintain the following invariants:
	\begin{enumerate}
		\item The set $\overline{\psi_\ell[T_\ell]}$ satisfies the conditions of Definition~\ref{def:balko} for all levels with the exception of $N_\ell-1$.
		\item If $\ell>0$, then, for every $u\in T_\ell$, the word $\psi_\ell(u)$ extends $\psi_{\ell-1}(u|_{\ell-1})$.
	\end{enumerate}

	We let $N_0=0$ and put $\psi_0$ to map the empty word to the empty word.
	Now, assume that $N_{\ell-1}$ and $\psi_{\ell-1}$ are already defined.
	We inductively define a sequence of functions $\psi^i_\ell\colon T_\ell\to \Sigma^*_{N_{\ell-1}+i}$. Put $\psi^0_\ell(u)=\psi_{\ell-1}(u|_{\ell-1})$.
	Now, we proceed in steps. At step $j$, apply the first of the following constructions that can be applied and terminate the procedure if none of them applies:
	\begin{enumerate}
		\item If there are distinct words $w,w'$ from $T_\ell$ such that $\psi^{j-1}_\ell(w)=\psi^{j-1}_\ell(w')$ and $w'|_{j-1}=w|_{j-1}$, we construct $\psi^j_\ell$ by extending each word in $\psi^{j-1}_\ell[T_\ell]$ by an additional character so that the conditions on the branching of $\psi^{j-1}_\ell(w)$ in Definition~\ref{def:balko} are satisfied.
		\item If there are words $w$ and $w'$ with $w\lexlt w'$ such that $w\perp w'$ and $\psi^{j-1}_\ell(w)\not \perp\psi^{j-1}_\ell(w')$ and condition  \ref{A2} is satisfied for the value range of~$\psi^{j-1}_\ell$, we construct $\psi^j_\ell$ to satisfy the conditions on new $\perp$ for $\psi^{j-1}_\ell(w)$ and $\psi^{j-1}_\ell(w')$ as given by Definition~\ref{def:balko}.
		\item If there are words $w$ and $w'$ with $w\lexlt w'$ such that $w\prec w'$ and $\psi^{j-1}_\ell(w)\not \prec\psi^{j-1}_\ell(w')$ and conditions \ref{B1} and \ref{B2} are satisfied for the value range of~$\psi^{j-1}_\ell$, we construct $\psi^j_\ell$ to satisfy the conditions on new $\prec$ for $\psi^{j-1}_\ell(w)$ and $\psi^{j-1}_\ell(w')$ as given by Definition~\ref{def:balko}.
	\end{enumerate}

	Let $J$ be the last index $j$ for which $\psi^j_\ell$ is defined. It is possible to prove that
	$\psi^J_\ell$ is actually an isomorphism  $\AmbStr{T_\ell}\to
		\AmbStr{\psi^J_\ell[T_\ell]}$. To do so, we need to verify that the procedure
	does not terminate early. Clearly, all the branching can be realized since there are no conditions on step~1. We also have $\psi^J_\ell(w)\perp \psi^J_\ell(w')\implies w\perp w'$ and $\psi^J_\ell(w)\preceq \psi^J_\ell(w')\implies w\preceq w'$ for $w,w'\in T_\ell$. To see that these implications are in fact equivalences, we define the \emph{distance} of $w$ and $w'$ as $|\{u:u\in T_\ell,w\lexlt u\lexleq w'\}|$.
	By Proposition~\ref{prop:perp}, one can add all pairs to the relation $\perp$ in the order of increasing distances to ensure that condition  \ref{A2} is satisfied and then it is possible to add all pairs of $\preceq$ in the order of decreasing distances so that conditions \ref{B1} and \ref{B2} are satisfied.

	Finally, we put $N_\ell=|\psi^J_\ell(w)|$ for some $w\in T_\ell$ and $\psi_\ell=\psi^J_\ell$.
	Once all functions $\psi_\ell$ are constructed, we can set $\psi(i)=\psi_i(i)$.\qed
\end{proof}

Given $S\subseteq \Sigma^*$, we call a level $\overline{S}_i$ \emph{interesting} if the structure $\overline{\str{S}}_i=(\overline{S}_i,\lexleq,\preceq\nobreak,\allowbreak\perp)$ is not isomorphic to  $\overline{\str{S}}_{i+1}=\AmbStr{\overline{S}_{i+1}}$
or there exist incompatible $u,v\in S_{i+1}$ such that $u|_{i}$ and $v|_{i}$ are compatible.
Let $I(S)$ be the set of all interesting levels in $\overline{\str{S}}$. Let $\tau_S\colon\overline{S}\to \Sigma^*$ be a mapping assigning every $w\in S$ the word created from $w$ by deleting all characters with indices not in $I(S)$. Define $\tau(S)=\tau_S[S]$ and call it the \emph{embedding type of $S$}. The following observation, which is a direct consequence of Definition~\ref{def:balko}, establishes that a sub-type of a poset-type is also a poset-type.

\begin{observation}\label{obs:types}
	For a poset-type $S$ and $S'\subseteq L(S)$, $\tau(\overline{S'})=\overline{\tau(S')}$ is a poset-type.\qed
\end{observation}

Given a finite partial order $\str{A}$, we construct a function (colouring) $\chi_\str{A}\colon\binom{\str{P}}{\str{A}}\to T(\str{A})$ by setting $\chi_\str{A}(f)=\tau(\psi[f[A]])$ for every $f\in \binom{\str{P}}{\str{A}}$.
We show that $\chi_\str{A}$ is an \emph{unavoidable coloring} in the following sense, which then implies Theorem~\ref{thm:main}.

\begin{theorem}
	\label{thm:embthm}
	For every finite partial order $\str{A}$ and every $f\in \binom{\str{P}}{\str{P}}$, we have $\left\{\chi_\str{A}[f\circ g]: g\in \binom{\str{P}}{\str{A}}\right\}=T(\str{A})$.
\end{theorem}

We outline the proof of Theorem~\ref{thm:embthm} in the rest of this Section.
We fix $f\in\binom{\str{P}}{\str{P}}$ and an arbitrary embedding $f'\colon(\Sigma^*,\preceq)\to f[\str{P}]$ (which exists, as every countable partial order embeds to $\str{P}$).
We set $h=\psi\circ f'$ and observe that it is an embedding $(\Sigma^*,\preceq)\to (\Sigma^*,\preceq)$. By Theorem~\ref{thm:posetemb} and Observation~\ref{obs:types}, we know that $\tau(h(\Sigma^*))$ is a poset-type and images of $\Sigma^*$ correspond to its leafs.  We will show a way to embed an arbitrary poset-type to $\tau(h(\Sigma^*))$.

We adapt a proof by Laflamme, Sauer and Vuksanovi\'c~\cite{Laflamme2006}. A word $u\in\Sigma^*$ is a \emph{successor} of $w\in\Sigma^*$ if $|u|\geq|w|$ and $u|_{|w|}=w$. A subset $A\subseteq\Sigma^*$ is \emph{dense above $u$} if every successor $u'$ of $u$ has a successor in $A$.
Given $u,v\in \Sigma^*$, we say that $v$ is \emph{$u$-large} if the set $\{u':h(u')\hbox{ is a successor of }v\}$ is dense above $u$.
\begin{observation}
	\label{obs:extend}
	Let $u,v\in G$ be vertices such that $v$ is $u$-large. Then
	\begin{enumerate}[label=(\roman*)]
		\item\label{item:u} $v$ is $u'$-large for every successor $u'$ of $u$, and
		\item\label{item:v} for every $\ell>|v|$ there exists a successor $v'$
		of $v$ of length $\ell$ and a successor $u'$ of $u$ such that $v'$ is $u'$-large.\qed
	\end{enumerate}
\end{observation}
\begin{proof}[of Theorem~\ref{thm:embthm}, Sketch]
	We construct functions $\alpha\colon \Sigma^*\to\Sigma^*$, $\beta\colon \Sigma^*\to\Sigma^*$  and sequences of integers $M_i$ and $N_i$, $i\in\omega$, satisfying:
	\begin{enumerate}[label=(\Roman*)]
		\item\label{I} For every $i\in \omega$, $u,v\in \Sigma^*_i$ it holds that $M_i< |\alpha(u)|\leq M_{i+1}$, $N_i< |\beta(u)|\leq|h(\alpha(u))|\leq N_{i+1}$ and if $u\lexlt v$, then $|h(\alpha(u))|<|\beta(v)|$.
		\item\label{II} For every $u,v\in \Sigma^*$ and every $j<\min(|\alpha(u)|,|\alpha(v)|)$ such that $\alpha(u)_j\neq \alpha(v)_j$ there exists $i<\min(|u|,|v|)$ satisfying $M_i=j$, $\alpha(u)_j=u_{i-1}$ and $\alpha(v)_j=v_{i-1}$.
		\item\label{III} For every $u\in \Sigma^*$ it holds that $\beta(u)$ is $\alpha(u)$-large and $h(\alpha(u))$ is a successor of $\beta(u)$.
		\item\label{IV} For every $u,u'\in \Sigma^*$ such that $u$ is a successor of $u'$ it holds that $\alpha(u)$ is a successor of $\alpha(u')$ and $\beta(u)$ is a successor of $\beta(u')$.

	\end{enumerate}
	Functions $\alpha$ and $\beta$ are built by an induction on levels of $\Sigma^*$ by repeated applications of Observation~\ref{obs:extend}. By~\ref{III} the partial maps $\alpha,\beta$ always extend. Moreover:
	\begin{claim}\label{claim}
		For every poset-type $S$ and every $i>0$, the structures
		$\str{S}_i=\AmbStr{S_i}$,
		$\str{S}'_i=\AmbStr{\alpha[S_i]}$ and $\str{S}''_i=\AmbStr{\beta[S_i]}$ are mutually isomorphic.
	\end{claim}
	Using the Claim, it is possible to verify that for every poset-type $S$ it holds that $\tau(h[\alpha[L(S)]])=S$.
	From this, Theorem~\ref{thm:embthm} follows: For every type $S\in T(\str{A})$ we have that $h(\alpha(L(S)))$ gives a copy of $\str{A}$ within $f(\str{P})$ of the given type.
\end{proof}

\paragraph{Acknowledgement}
J.~H. and M.~K. are supported by the project 21-10775S of  the  Czech  Science Foundation (GA\v CR). This article is part of a project that has received funding from the European Research Council (ERC) under the European Union's Horizon 2020 research and innovation programme (grant agreement No 810115).
M.~B. and J.~H. were supported by the Center for Foundations of Modern Computer Science (Charles University project UNCE/SCI/004). M.~K. was supported by the Charles University Grant Agency (GA UK), project 378119.
N.~D.\ is supported by National Science Foundation grant DMS-1901753. A.~Z.\ was supported by National Science Foundation grant DMS-1803489.

\end{document}